\documentclass[12pt,oneside,reqno]{amsart}
\usepackage[colorlinks,linkcolor=black,citecolor=blue]{hyperref}
\usepackage{amssymb}
\usepackage{}
\usepackage{txfonts}
\usepackage{bbm}
\usepackage{cases}
\usepackage{amsmath}
\usepackage{graphicx}
\usepackage{mathrsfs}
\usepackage{stmaryrd}
\usepackage{amsfonts}
\usepackage{comment}
\usepackage{enumerate,amsmath,amssymb,amsthm}

\numberwithin{equation}{section}

\newcommand{\be}{\begin{eqnarray}}
\newcommand{\ee}{\end{eqnarray}}
\newcommand{\ce}{\begin{eqnarray*}}
\newcommand{\de}{\end{eqnarray*}}
\newtheorem{theorem}{Theorem}[section]
\newtheorem{lemma}[theorem]{Lemma}
\newtheorem{remark}[theorem]{Remark}
\newtheorem{definition}[theorem]{Definition}
\newtheorem{prop}[theorem]{Proposition}
\newtheorem{example}[theorem]{Example}
\newtheorem{corollary}[theorem]{Corollary}

\def\eps{\varepsilon}

\def\a{\alpha}

\def\d{\delta}

\def\[{{\Big[}}
\def\]{{\Big]}}
\def\<{{\langle}}
\def\>{{\rangle}}
\def\({{\Big(}}
\def\){{\Big)}}

\def\bx{{\mathbf{x}}}

\def\dif{{\mathord{{\rm d}}}}

\def\min{{\mathord{{\rm min}}}}

\def\no{\nonumber}
\def\={&\!\!=\!\!&}

\def\cK{{\mathcal K}}

\def\mR{{\mathbb R}}
\def\mS{{\mathbb S}}

\def\bE{{\mathbf E}}
\def\1{{\mathbf{1}}}

\def\sB{{\mathscr B}}

\def\sK{{\mathscr K}}

\def\sP{{\mathscr P}}

\def\E{\mathbb E}

\def\geq{\geqslant}
\def\leq{\leqslant}

\def\bd{d}

\def\eps{\varepsilon}

\def\a{\alpha}

\def\d{\delta}

\def\[{{\Big[}}
\def\]{{\Big]}}
\def\<{{\langle}}
\def\>{{\rangle}}
\def\({{\Big(}}
\def\){{\Big)}}

\def\bx{{\mathbf{x}}}

\def\dif{{\mathord{{\rm d}}}}

\def\min{{\mathord{{\rm min}}}}

\def\no{\nonumber}
\def\={&\!\!=\!\!&}

\def\bt{\begin{theorem}}
\def\et{\end{theorem}}
\def\bl{\begin{lemma}}
\def\el{\end{lemma}}
\def\br{\begin{remark}}
\def\er{\end{remark}}
\def\bx{\begin{Examples}}
\def\ex{\end{Examples}}
\def\bd{\begin{definition}}
\def\ed{\end{definition}}
\def\bp{\begin{proposition}}
\def\ep{\end{proposition}}
\def\bc{\begin{corollary}}
\def\ec{\end{corollary}}
\def\bpf{\begin{proof}}
\def\epf{\end{proof}}

\def\geq{\geqslant}
\def\leq{\leqslant}

 \def\R{\mathbb R}
 \def\R{\mathbb R}

\def\<{\langle} \def\>{\rangle} \def\GG{\Gamma} \def\gg{\gamma}
    
\def\d{\text{\rm{d}}}   
   
 \def\beq{\begin{equation}}

 \def\P{\mathbb P}

\allowdisplaybreaks

\begin{document}
\title{\bf{Weak uniqueness for SDEs driven by supercritical stable processes with H\"older drifts}}

\author{Guohuan Zhao}

\address{Guohuan Zhao:
Applied Mathematics, Chinese Academy of Science,
Beijing, 100081, P.R.China\\
Email: zhaoguohuan@gmail.com
 }

\thanks{Research of GH is partially supported by National Postdoctoral Program for Innovative Talents (201600182) of China.}

\begin{abstract}
In this paper, we investigate stochastic differential equations(SDEs) driven by a class of  supercritical $\a$-stable process(including the rotational symmetric $\a-$stable process) with drift $b$. The weak well-posedness is proved, provided that  the $(1-\a)$-H\"older semi-norm of $b$ is sufficient small.


\end{abstract}

\maketitle
\section{Introduction and main result}
The main purpose of this paper is to establish the weak well-posedness for   stochastic differential equations
driven by a class of symmetric supercritical $\alpha$-stable process with non-Liptchitz drift $b$, here supercritical means the index $\a\in(0,1)$. More precisely, we are concerning with the following SDE:

\begin{equation}\label{SDE}
\d X_t=b(X_t)\d t+\d Z_t;\quad X_0=x \in \R^d, 
\end{equation}
here $b:\mR^d\to\mR^d$ is a Borel measurable function, and $Z$ is an $\a$-stable process with L\'evy measure $\mu$. 
$$
\mu(A)=\int^\infty_0\left(\int_{\mS^{d-1}}\frac{1_A (r\theta)\Sigma(\dif\theta)}{r^{1+\alpha}}\right)\dif r,\quad A\in\sB(\mR^d),
$$
where $\Sigma$ is a finite measure over the unite sphere $\mS^{d-1}$ in $\R^d$.

When $Z_t$ is a Brownian motion (which corresponds to $\a=2$), the weak uniqueness of \eqref{SDE} can be proved by well know Girsanov transform. However, things become quite different when $Z$ is a pure jump process. When  $Z$ is a rotational symmetric $\a$-stable process and $\a\in(1,2)$, using heat kernel estimates and perturbation argument, in \cite{Ch-Wa} Chen and Wang proved the well-posedness for \eqref{SDE}  even when the drift term $b$ belongs to a certain Kato class of the rotationally symmetric $\a-$stable process $Z$ on $\R^d$. In \cite{tanaka1974perturbation},  Tanaka etc. showed that if $Z$ is a 1-dimensional supercritical symmetric $\alpha-$stable process and $b\in C^\beta$ with $\beta>1-\alpha$, then weak well-posedness  holds. Simultaneously, pathwise uniqueness follows by Proposition 1.1 of  their work.  However, for the same noise, when $\beta<1-\alpha$, they showed that if $b(x)=(1\wedge |x|^\beta)I_{\{x\geq0\}}-(1\wedge |x|^\beta)I_{\{x<0\}}$ and $X_0=0$, the weak uniqueness fails. When $d\geq 1$ and $\beta=1-\alpha$, as far as to our knowledge, there is no results until now. The main purpose of this work is trying to give a partial answer to this.  

\medskip 

Unfortunately, we can not prove the result for all non-degenerate $\a-$stable processes. An technical assumption on the L\'evy measure $\mu$ of $Z$ is needed in this paper: 
\begin{enumerate}[{\bf (H$^\theta_\kappa$)}]
\item There are two constant $\kappa>0$ and $\theta\in\left(\arccos\big(\sqrt{1/(2-\a)}\big), \pi/4\right)$ such that for any vector $d\in \R^d$, one can find a  cone $$S(\pmb{n}_{d},\theta):=\left\{x:  \<x, \pmb{n}_d\>>|x|\cos (\tfrac{\theta}{2}) \right\}$$  with vertex $0$, apex angle $\theta$ and symmetry axis $\pmb{n}_d\in \mS^{d-1}$ containing $d$ and for any $\delta\in (0,1)$, 
$$\int_{\{ y\in B_\delta\cap S(\pmb{n}_{d},\theta)\}}|y|^2\mu(\d y)\geq \kappa \delta^{2-\alpha}. $$
\end{enumerate}

Our main result is 
\begin{theorem}\label{Weak uniqueness}
Suppose $\mu$ is symmetric and satisfies {\bf (H$^{\theta}_\kappa$)},   then there exists a constant $\eps_0>0$ such that  if $b$ is bounded and 
\begin{equation}\label{b-condition}
\lim_{\delta\downarrow 0}\sup_{0<|x-y|<\delta }\frac{|b(x)-b(y)|}{|x-y|^{1-\alpha}}<\eps_0, 
\end{equation}
then \eqref{SDE} has a unique weak solution.
\end{theorem}
Roughly speaking, our result says if $\mu$ satisfies {\bf (H$^\theta_\kappa$)} and the $(1-\a)$-H\"older semi-norm of $b$ is sufficient small, then \eqref{SDE} is well-posed. 
Before motiving on, let us give two examples.
\begin{example}
Suppose $d\geq 1$, $b(x)=1\wedge|x|^{1-\alpha}|\log|x||^{-1}$, $Z_t$ is an rotational symmetric $\a$-stable process in $\R^d$ i.e. $\mu(\d z)=|z|^{-d-\a}\d z$. It is easy to verify that $\mu$ satisfies {\bf {H}$^\theta_\kappa$} and $b$ satisfies \eqref{b-condition}, but for any $\beta>1-\a$, b is not in the H\"older space $C^\beta$. 
\end{example}

\begin{example}
Suppose $d=2$, 
$$\mu(dx)= \sum_{i=0}^{[\pi/\vartheta]}\1_{\{x_2=\tan (i\vartheta)x_1\}} \frac{dx_1}{\cos(i\vartheta)|x|^{1+\alpha}}, $$
here $\frac{\pi}{2\vartheta}\notin \mathbb{N}$ and $0<\vartheta<2\arccos\sqrt{\frac{1}{2-\alpha}}$, then we can also verify that $\mu$ satisfies {\bf {H}$^\theta_\kappa$}(Notice that in this case $\mu$ is very singular). But unfortunately,  if the support of the singular L\'evy measure is ``sparser", for instance  
$$\mu(dx)=\1_{\{x_2=0\}}\frac{dx_1}{|x_1|^{1+\alpha}}+\1_{\{x_1=0\}}\frac{dx_2}{|x_2|^{1+\alpha}} $$
i.e. $Z_t=(Z^1_t, Z^2_t)$ with $Z^1, Z^2$ are two independent one dimensional symmetric $\a$-stable processes, then the L\'evy process $\mu$ does not  satisfy the assumption {\bf {H}$^\theta_\kappa$}. 

\end{example}

\begin{remark}
The following two problems are quite interesting and we do not know the answers.  
\begin{enumerate}
  \item If $\mu(\d z)=|z|^{-d-\a}\d z$, is there a function $b\in C^{1-\alpha}(\R^d)$ such that  weak uniqueness for \eqref{SDE} fails?
  \item For any $d\geq 2$ and $\mu(\d z)=\sum_{i=1}^d \1_{\{z_1=0,\cdots,\hat z_i=0,\cdots,z^n=0\}} |z_i|^{-d-\a} {\d z_i}$,  can one show the weak uniqueness?
\end{enumerate}
\end{remark}

\medskip
Now let us give a brief introduce to our approach in this work. Since we assume the coefficient $b$ is continuous, one can get the weak existence of \eqref{SDE} easily. So the key is the uniqueness. It is  well know that the weak solution to SDEs are close relate to martingale problem(see Definition \ref{Def[MP]} below).  Denote $A:=\mathcal{L}+b\cdot\nabla,$ where $\mathcal{L}$ is the generator of $Z_t$, define 
$$
\Gamma_x:=\left\{ \hbox{ all the solutions to martingale problem}\ (A,\delta_x)\right\}. 
$$
By  Theorem 4.4.2 of \cite{ethier2009markov}, in order to prove the uniqueness of weak solution, one just need to show: for any $\lambda>0,~x\in\R^d,~f\in C^\infty_c$ and $\P_x,~ \P'_x\in \Gamma_x:=\big\{$ all solutions to martingale problem $(A,\delta_x)$ $\big\}$
\begin{equation}\label{eql}
\bE^{\P_x}\left[\int_0^\infty e^{-\lambda t}f(X_t)dt\right]=\bE^{\P'_x} \left[\int_0^\infty e^{-\lambda t}f(X_t)dt\right]. 
\end{equation}

When $b$ is smooth,  $\GG_x$ contains only one element $\P_x$, and $u(x)=\P_x\left[\int_0^\infty e^{-\lambda t}f(X_t)dt\right]$ is the unique classical solution of following resolvent equation:
\begin{equation}\label{Resolvent eq}
\lambda u-Au=f. 
\end{equation}
Conversely, if the above equation has a classical solution, then the well-posedness of martingale problem holds. 
In PDE literatures, the existence and uniqueness of classical solutions to linear equations obtained by priori estimates. In general, these estimates in particular function space are not easy, especially for nonlocal equations. There are several references on this topic, for instants, \cite{silvestre2010differentiability}, \cite{Ch-So-Zh} and \cite{Ch-Zh-Zh}. In 1980's Michael G Crandall and Pierre-Louis Lions introduced the concept of viscosity solution. Generally, the existence and uniqueness of this kind of weak sense solutions can be obtained as soon as one get the comparison principle (see \cite{crandall1992user} for details). The development of viscosity solutions for nonlocal nonlinear equations can be found in \cite{barles2008second} and \cite{barles2010h}. The closed relationship between viscosity solution and martingale problem were studied in \cite{Fe-Ku} and \cite{costantini2014viscosity}.  Our method is mainly  inspired by \cite{costantini2014viscosity},  in section 3, we will establish the comparison principle for subviscosity solutions and superviscosity solutions to resolvent equation \eqref{Resolvent eq} under some technical assumptions which we believe can be reduced, then by the similar argument in \cite{costantini2014viscosity}, we show  \eqref{eql} holds, using this we get the weak uniqueness of \eqref{SDE}. And in order to keep the statement simple,   we assume the process $Z_t$ is symmetric $\a-$stable process for simple, but the methods used here can be applied to more general case, even for $Z_t$ is a jump Markov process satisfying some suitable conditions.

\medskip
Let us also mention that there are many literatures study strong solution to \eqref{SDE}. Priola proved the strong existence and uniqueness for \eqref{SDE} in \cite{Pr1} by using Zvonkin's transform, under the following conditions:  $Z_t$ is a rotational symmetry $\alpha-$stable processes with $\alpha\geq 1$ and $b\in C^\beta $ with $\beta>1-\frac{\alpha}{2}$. 
In \cite{Zh1}, Zhang considered the case that $b$ belongs to some fractional Sobolev spaces. For a large class of L\'evy processes, Chen, Song and Zhang  proved established strong existence and pathwise uniqueness for \eqref{SDE}  in their work \cite{Ch-So-Zh}, when $b$ is time dependent, H\"older continuous in $x$. Very recently, the similar result was improved in \cite{Ch-Zh-Zh}. Therein,
the authors not only extend the main result of \cite{Pr1} and \cite{Pr2} for the subcritical and critical case ($\alpha \in [1,2)$) to more general L\'evy processes and time-dependent drifts, but also establish strong existence and pathwise uniqueness for the supercritical case.

\medskip

We close this section by mentioning 
some conventions used throughout this paper:
We use $:=$ as a way of definition. For $a, b\in \mR$, $a\vee b:= \max \{a, b\}$ and $a\wedge b:=\min \{a, b\}$,
The letter $c$ or $C$ with or without subscripts stands for an unimportant constant, whose value may change in difference places.

\section{Some definitions}
In this section, we give some important definitions that will be used later. Let $D(\R_+; \R^d)$ be the c\`adl\`ag space with Skorokhord topology and  denote $X_t(\omega)=\omega(t)$ for any $\omega\in D(\R_+; \R^d)$ hereinbelow. The following is the definition of martingale problem.
\begin{definition}\label{Def[MP]}
 Given a linear operator $A:C^2_b(\R^d)\to C_b(\R^d)$ and a probability measure $\nu$ on $\R^d$. Martingale problem for $(A, \nu)$ consists in finding a probability measure $\P$  over the c\`adl\`ag space $D(\R_+; \R^d)$ such that for any $B\in \sB(\R^d)$, $\P(X_0\in B)=\nu (B)$ and whenever $f\in C^2_b(\R^d)$, we have that
$$
t\mapsto f(X_t)-f(X_0)-\int_0^t Af(X_s)\d s
$$
is a  martingale under $\P$. 
\end{definition}

Next, we introduce the definition of viscosity solution for general resolvent equation.
\begin{definition}\label{Def[CK]}
\begin{enumerate}
  \item An upper semi-continuous bounded function $u$ is a viscosity subsolution to $\lambda u-Au=f$ (or $\lambda u-Au\leq f$) if for any test function $\phi\in C^2_b(\R^d)$ touches $u$ from above at $x_0$, 
$$\lambda u(x_0)-A\phi(x_0)\leq f(x_0).$$
  \item A lower semi-continuous bounded function $v$ is a viscosity supersolution to $\lambda v-Av=f$(or $\lambda v-Av\geq f$) if for any test function $\phi\in C^2_b(\R^d)$ touches $v$ from below at $x_0$ and $$\lambda v(x_0)-A\phi(x_0)\geq f(x_0).$$
\end{enumerate}
\end{definition}
From the next easy lemma, we can see the close relationship between solutions to the martingale problem and viscosity solution to related resolvent equations.
\begin{lemma}
For any $\lambda>0$ and $ x\in \R^d$, suppose $(X_t,\P_x)$ is a martingale solution to $(A,\delta_x)$.
\begin{enumerate}
  \item[(i)] 
If $u\in C^2_b(\R^d)$ is a classical solution to \eqref{Resolvent eq}, then $$t\mapsto u(X_t)-\int_0^t \left[\lambda u(X_s)-f(X_s)\right]\d s$$
is a martingale under $\P_x$.
  \item[(ii)] $u, f\in C_b(\R^d)$ and the process $$t\mapsto u(X_t)-\int_0^t \left[\lambda u(X_s)-f(X_s)\right]\d s$$
is a martingale under $\P_x$. Then $u$ is a viscosity solution to \eqref{Resolvent eq}.
\end{enumerate}
\end{lemma}
\begin{proof}
\begin{enumerate}
\item By the definition of martingale problem, 
$$
t\mapsto u(X_t)-\int_0^t Au(X_s)\d s=u(X_t)-\int_0^t \left[\lambda u(X_s)-f(X_s)\right]\d s
$$
is a martingale. 
\item Suppose $\varphi\in C^2_b(\R^d)$ touches $v$ from above at $x_0$, then
$$0\geq t^{-1} \E_{x_0}[u(X_t)-\varphi(X_t)]=t^{-1} \E_{x_0}\int_0^t \lambda u(X_s)-f(X_s)-A\varphi(X_s)\d s$$
Let $t\rightarrow 0$, we get
$$\lambda u(x_0)-A\varphi(x_0)\leq f(x_0)$$
By this, we get $u$ is a viscosity subsolution. The some argument shows $u$ is also a viscosity supersolution, so we complete our proof.
\end{enumerate}
\end{proof}

We need anther definition of viscosity sub(super)solution for later use. Let us introduce two auxiliary operators: $\mathcal{L}^\delta$, $\mathcal{L}_{\delta}$, 
\begin{align*} 
&\mathcal{L}^\delta\phi(x):=\int_{|y|\leq \delta}[\phi(x+y)-\phi(x)-\nabla\phi(x)\cdot y 1_{B_1}(y)]\mu(\d y); \\
&\mathcal{L}_\delta\phi(x):=\int_{|y|> \delta}(\phi(x+y)-\phi(x))\mu(\d y). 
\end{align*}
The following definition can be find in \cite{barles2008second}.

\begin{definition}\label{Def[BI]}
\begin{enumerate}
  \item An upper semi-continuous bounded function $u$ is a viscosity subsolution to $\lambda u-Au=f$(or $\lambda u-Au\leq f$) if for any test function $\phi\in C^2(B(x_0,\delta))$, $x_0$ is a local maximum of $u-\phi$ in $B(x_0,\delta)$, then  $$\lambda u(x_0)-b(x_0)\cdot\nabla\phi(x_0)-\mathcal{L}^\delta\phi(x_0)-\mathcal{L}_\delta u(x_0)\leq f(x_0).$$
  \item A lower semi-continuous bounded function $v$ is a viscosity supersolution to $\lambda v-Av=f$(or $\lambda v-Av\leq f$) if for any test function $\phi\in C^2(B(x_0,\delta))$, $x_0$ is a local minimum of $v-\phi$ in $B(x_0,\delta)$, then  $$\lambda v(x_0)-b(x_0)\cdot\nabla\phi(x_0)-\mathcal{L}^\delta\phi(x_0)-\mathcal{L}_\delta v(x_0)\geq f(x_0).$$
\end{enumerate}
\end{definition}

Next we prove the equivalence of above two definitions of viscosity solution. 
\begin{prop}\label{Equivalent}
Definitions \ref{Def[CK]} and \ref{Def[BI]} are equivalent for bounded solutions.
\end{prop}
\begin{proof}
\begin{enumerate}
\item[(i)]  Suppose $u$ is a bounded viscosity subsolution of resolvent equation in the sense of Definition \ref{Def[CK]}, and $\phi\in C^2(B(x_0,\delta))$ touches $u$ at $x_0$ from above. We will construct a sequence $\{\phi_k\} \subset  C^2_b(\R^d)$, which is uniformly bounded in $C^2_b(\R^d)$ such that
\begin{itemize}
\item $u-\phi_k$ attains a global maximum at $x_0$;
\item  $\mathcal{L}^\delta\phi_k(x_0)\rightarrow \mathcal{L}^\delta\phi(x_0)$, as $k\rightarrow \infty$;
\item  $\nabla\phi_k(x_0)\rightarrow \nabla\phi(x_0)$, as $k\rightarrow \infty$;
\item  $\mathcal{L}_\delta\phi_k(x_0)\rightarrow \mathcal{L}_\delta u(x_0)$, as $k\rightarrow \infty$.
\end{itemize}
Indeed, we can assume $x_0=0$. Since $u$ is bounded upper semi-continuous, there exits a sequence of $C^2_b$ functions $\{u_k\}$ which is uniformly bounded and $u_k(x)\geq u(x)+\frac{1}{k}$ such that $u_k\downarrow u$. Moreover, we can find positive constants $r_k\downarrow 0$ such that 
$$\sup_{x\in\partial B_\delta; |x-y|\leq r_k}(u(y)-u(x))\leq(1\land \delta^2)/(2k)$$
Let $\rho\geq 0$, $\rho(x)\in C^\infty([0,1])$ with $\rho(0)=1, ~\rho'(0)=\rho''(0)=\rho(1)=\rho'(1)=\rho''(1)=0$ . Define  
$\phi_k(x)=\phi(x)+\frac{1}{k}|x|^2$, if $|x|\leq \delta$; $\phi_k(x)=u_k(x)$ if $|x|>\delta+r_k$ and  
$$\phi_k(x)=\phi_k(x\delta/|x|)\cdot\rho\big((|x|-\delta)/r_k\big)+u_k\big((\delta+r_k)x/|x|\big)\cdot \big[1-\rho\big((|x|-\delta)/r_k\big)\big]$$
if $\delta\leq|x|\leq\delta+r_k$. It is not hard to verify that $\phi_k$ satisfies all the properties list above. By Definition \ref{Def[CK]}, we obtain $$\lambda u(x_0)-b(x_0)\cdot\nabla\phi_k(x_0)-\mathcal{L}^\delta\phi_k(x_0)-\mathcal{L}_\delta \phi_k(x_0)\leq f(x_0).$$
Let $k\rightarrow \infty$, we get 
$$\lambda u(x_0)-b(x_0)\cdot\nabla\phi(x_0)-\mathcal{L}^\delta\phi(x_0)-\mathcal{L}_\delta u(x_0)\leq f(x_0).$$
\item[(ii)]  Suppose $u$ is a bounded viscosity subsolution of resolvent equation in the sense of Definition \ref{Def[BI]}. If $\phi\in C_b^2$, $u-\phi$ reaches it's global maximum $0$ at $x_0$, then 
$$\lambda u(x_0)-b(x_0)\cdot\nabla\phi(x_0)-\mathcal{L}^\delta\phi(x)-\int_{|y|>\delta} [u(x_0+y)-u(x_0)]\mu(\d y)\leq 0$$
Since $\phi(x_0+y)\geq u(x_0+y),~\phi(x_0)=u(x_0)$, obviously we have 
$$\lambda u(x_0)-b(x_0)\cdot\nabla\phi(x_0)-\mathcal{L}^\delta\phi(x)-\int_{|y|>\delta} [\phi(x_0+y)-\phi(x_0)]\mu(\d y)\leq0$$
\end{enumerate}
\end{proof}

\section{Well-poseness of martingale problem}
In this section, we will give the proof of our main result. Before that we need to introduce some notations. For any metric space $(S,d)$, let $\sP(S)$ be the collection of all probability measures on $S$.  We denote 
$$
\Gamma:=\left\{\P: \P \ \mbox{is a solution to martingale problem}\  A \right\};$$
and for any $\nu\in \sP(\R^d)$, denote 
$$
\Gamma_\nu:=\left\{\P: \P \ \mbox{is a solution to martingale problem}\  (A,\nu)\right\}; \quad \Gamma_x:=\Gamma_{\delta_x}.
$$

The following lemma is just a corollary of \cite[Lemma 3.5]{costantini2014viscosity}. 
\begin{lemma}\label{sub-super}
Suppose $b$ is bounded continous function, then for any $f\in C_b(\R^d)$ the following two functions  
\begin{align} \label{u&v}
u(x):=\sup_{\P\in\Gamma_x}\bE^{\P} \left[\int_0^\infty e^{-\lambda t}f(X_t)\d t\right], \quad v(x):=\inf_{\P\in\Gamma_x}\bE^\P \left[\int_0^\infty e^{-\lambda t}f(X_t)\d t\right]. 
\end{align} 
are bounded subsolution and bounded supersolution to the resolvent equation, respectively.
\end{lemma}
\begin{proof}
By Lemma 3.5 of \cite{costantini2014viscosity}, we only need to verify the condition 3.1(c) of \cite{costantini2014viscosity}. For any $\cK\subseteq \sP(\R^d)$ is compact and $\eps>0$ we can find $K\subseteq \R^d$ is compact such that 
\be\label{EE22}
\inf_{\P\in \cup_{\nu\in \cK}  \Gamma_{\nu}}\P(X_0\in K)=\inf_{\nu\in \cK}\nu(K)\geq 1-\tfrac{\eps}{2}. 
\ee
For any $\P\in \bigcup_{\nu\in \cK}  \Gamma_{\nu}$, 
$$\P\left(X.-X_0-\int_0^{\cdot}b(X_s)\d s\in B\subseteq D(\R_+; \R^d)\right)=\P(Z.\in B), $$
noticing the distribution of $Z.$ under $\P$ is the one of rotational symmetric $\a-$stable process starts from $0$, so there exists a compact set $\sK_1\subseteq D(\R_+; \R^d)$ such that
\begin{align}\label{EE23}
\inf_{\P\in \cup_{\nu\in \cK}  \Gamma_{\nu}}\P\left(X.-X_0-\int_0^{\cdot}b(X_s)\d s\in \sK_1\right)\geq 1-\tfrac{\eps}{2}. 
\end{align}
Let $$\mathcal{C}=\left\{y\in C^1(\R_+; {\R^d}): y_0\in K, \|y'\|\leq \|b\|_\infty\right\}, $$ 
noticing 
\begin{align}\label{EE24}
\left\{X.: X_0\in K; X.-X_0-\int_0^{\cdot}b(X_s)\d s\in \sK_1\right\}\subseteq \sK:=\left\{X.: X.\in \mathcal{C}+\sK_1\right\},  
\end{align}
by Theorem 3.6.3 of \cite{ethier2009markov}, $\sK$ is relatively compact set in $D(\R_+; \R^d)$. Combining \eqref{EE22}, \eqref{EE23} and \eqref{EE24}, we get 
$$
\inf_{\P\in \cup_{\nu\in \cK}\Gamma_{\nu}} \P(\sK)\geq 1-\eps,
$$
which implies $\bigcup_{\nu\in \cK}  \Gamma_{\nu}$ is relatively compact in $\sP(D(\R_+; \R^d))$.  Now suppose $\{\P_n\}\subseteq \bigcup_{\nu\in \cK}  \Gamma_{\nu}$, then there exists a subsequence of $\{\P_n\}$, still denote by $\{\P_n\}$ for simple such that $\P_n\Rightarrow \P$ and $\P_n\in \Gamma_{\nu_n}$ with $\nu_n\in \cK$. It is not hard to see, $\{\nu_n\}$ has a unique limit point $\nu\in  \cK$ and one can verify $\P\in \Gamma_\nu$. So $\bigcup_{\nu\in \cK} \Gamma_\nu$ is compact, i.e. the condition 3.1(c) of \cite{costantini2014viscosity} holds.  

\end{proof}
We need the following auxiliary lemma for later use.  
\begin{lemma}\label{zero}
Suppose $f$ is a bounded function with compact support and $u, v$ are functions defined in \eqref{u&v}. Then,   
\begin{align}\label{decay}
\lim_{n\rightarrow \infty} \sup_{|x|\geq n}|v(x)|=\lim_{n\rightarrow \infty} \sup_{|x|\geq n}|u(x)|=0. 
\end{align}
\end{lemma}
\begin{proof}
By the above lemma, for any $x\in\R^d$, we can find $(\P_x, X)$ is the solution to the martingale problem $(A,\delta_x)$ and  
$$u(x)=\bE^{\P_x}\left(\int_0^\infty e^{-\lambda t}f(X_t)\d t\right). $$
Assume $f(x)=0, $ if $x\in B^c(0,n)$.  Let $g\in C^2(\R^d)$ satisfies $g(0)=0,~g(y)=1,$ if $y\in B^c(0,1)$.  Define $g_R^x(y):=g(\frac{y-x}{R})$. For any $\eps>0$, choose
\be\label{E-T} T>\left| \frac{\log(\lambda\eps/4\|f\|_\infty)}{\lambda}\right|,\ee
 and $m>0, R>0$, such that 
\be \label{E-m} T\int_{|z|>m}\mu(\d z)\leq \frac{\eps}{4\|g\|_\infty}, \ee
\be \label{E-R} R>\left(\frac{4T \|f\|_\infty\|\nabla^2g\|_\infty\int_{|y|\leq m}|y|^2\mu(\d y)}{\eps \lambda}\right)^{1/2}+\frac{4\|b\|_\infty\|\nabla g\|_\infty}{\eps}.\ee
Let $\tau_R=\inf\{t: X_t\notin B(x,R)\}$. By martingale property, we have 
$$\bE^{\P_x} g^x_R(X_{\tau_R\land T})-\bE^{\P_x} g^x_R(X_0)=\bE^{\P_x} \left[\int_0^{\tau_{R}\land T}Ag_R^x(X_s)\d s\right]. $$
By Taylor's expansion,  
\begin{align*}
A g_R^x(y)=&\frac{1}{2}\int_{\R^d}[g^x_R(y+z)-g^x_R(y)+ g^x_R(y-z)-g^x_R(y)]\mu(\d z)+b(y)\cdot \nabla g^x_R(y) \\
\leq&\|\nabla^2g_R^x\|_\infty \int_{|z|\leq m}|z|^2\mu(\d z)+\|g^x_R\|_\infty\int_{|z|>m}\mu(\d z)+\|b\|_\infty\|\nabla g^x_R\|_\infty\\
\leq&\frac{\|\nabla^2g\|_\infty}{R^2} \int_{|z|\leq m} |z|^2\mu(\d z)+\|g\|_\infty\int_{|z|>m}\mu(\d z)+\frac{\|b\|_\infty\|\nabla g\|_\infty}{R}. 
\end{align*}
By above inequality and \eqref{E-T}, \eqref{E-m}, \eqref{E-R}, 
\begin{align}
 \no \P_x({\tau_R}\leq T)\leq& \bE^{\P_x} g^x_R(X_{\tau_R\land T})=\bE^{\P_x} \int_0^{\tau_{R}\land T} Ag^x_R(X_s)\d s\\ \no
\leq& T\frac{\|\nabla^2 g\|_{\infty}}{R^2} \int_{|y|\leq m}|y|^2\mu(\d y)+T\|g\|_\infty\int_{|z|>m}\mu(\d z)+\frac{T\|b\|_\infty\|\nabla g\|_\infty}{R}\\
<&\frac{3\lambda\eps}{4\|f\|_\infty}\label{Tau-T}
\end{align}
For any fixed $x\in B^c(0, {n+R})$, since $f(X_t)=0$ if $t\leq T\wedge \tau_R$, we have 
$$|u(x)|=\int_0^\infty e^{-\lambda t} \bE^{\P_x} |f(X_t)|\d t\leq \P_x({\tau_R}\leq T) \frac{\|f\|_\infty}{\lambda}+\int_T^\infty e^{-\lambda t}\|f\|_\infty \d t.$$
Combining above inequlity and \eqref{E-T}, \eqref{Tau-T}, we get $|u(x)|<\eps$ for any $|x|>n+R$. So we complete our proof. 
\end{proof}
Now we give the comparison principle for the resolvent equation \eqref{Resolvent eq}, by which we can get the two functions defined in \eqref{u&v} are equivalent. 

\begin{lemma}\label{Comparison}
 Suppose $\mu$ satisfies {\bf (H$^\theta_\kappa$)}, then there exists a constant $\eps_0>0$ such that if $b$ satisfies 
 $$[b]_{1-\a}:=\sup_{x,y\in\R^d}\frac{|b(x)-b(y)|}{|x-y|^{1-\a}} < \varepsilon_0,$$ 
then for any subsolution $u$ and supersolution $v$ to the resolvent equation \eqref{Resolvent eq}, we have $u\leq v$, provided that $f\in C^1_c(\R^d)$ and $u,v$ satisfy \eqref{decay}. 
\end{lemma}
\begin{proof}
Suppose $u,v$ are subsoltion and supersolution to resolvent equation \eqref{Resolvent eq} respectively. Denote 
$$M: =\sup_{x\in\R^d} \{u(x)-v(x)\}.$$

{\bf Claim:} for any fixed $\gamma\in (\alpha,2-\tfrac{1}{\cos^2\theta})$, there exits a constant $\eps_0>0$ such that if $[b]_{1-\a}<\eps_0$,  then 
\begin{equation}\label{Holder}
u(x)-v(y)\leq M+ L|x-y|^\gamma, \end{equation}
for some sufficient large $L$. 
\medskip

Let us show the proof of \eqref{Holder}: fixing $\gg\in \big(\alpha,2-\tfrac{1}{\cos^2\theta}\big)$ and $\eta\in(0,1)$ such that 
\be\label{E-eta}
(1+\eta)^{\gamma-2}(2-\gamma)\cos^2\theta-(1-\eta)^{\gamma-2}>0. 
\ee
Assume \eqref{Holder} does not hold, then for any $L>0$, 
$$M_L:=\sup_{x,y\in \R^d}\{u(x)-v(y)-L|x-y|^{\gamma}\}-M>0. $$
Indeed, the superum can be reached at some point $(\hat{x},\hat{y})$.  Suppose $u(x_n)-v(y_n)-L|x_n-y_n|^{\gamma}-M\rightarrow M_L$, by \eqref{decay} and the assumption   $M_L>0$, we get $x_n, y_n$ must be bounded, say $|x_n|+|y_n|\leq K$. Since $u(x)-v(y)-L|x-y|^{\gamma}$ is upper semi-continuous, there exit $\hat{x}, \hat{y}\in \bar{B}_{K}$ such that $M_L=u(\hat{x})-v(\hat{y})-L|\hat{x}-\hat{y}|^{\gamma}-M$.  Now denote $d:=\hat{x}-\hat{y}\neq 0$.  Noticing $u$ and $v$ are bounded,  this  implies $d$ will go to $0$ as $L$ goes to $\infty$. Let  $\phi(x,y):=L|x-y|^\gg$, 
$\hat{\phi}(x):=\phi(x,\hat{y})$, $-\hat{\psi}(y):=-\phi(\hat{x},y)$ and $p:=\nabla\hat{\phi}(\hat{x})=\nabla(-\hat{\psi}(\hat{y}))=\gamma L |\hat{x}-\hat{y}|^{\gamma-2} (\hat{x}-\hat{y})=\gamma L |d|^{\gamma-2}d$.  For any $\delta'\leq \frac{1}{2}\eta |d|$,  noticing $\hat{\phi}(x)$ is smooth in $B(\hat{x}, \delta')$ and $-\hat{\psi}(y)$ is smooth near $B(\hat{y}, \delta')$, so by Definition \ref{Def[BI]}, we have 
\begin{equation}\label{Sum}
\begin{split}
0<\lambda (M_L+M)&\leq \lambda[u(\hat{x})-v(\hat{y})]\\
&\leq [f(\hat{x})+b(\hat{x})\cdot p+\mathcal{L}^{\delta'}\hat{\phi}(\hat{x})+\mathcal{L}_{\delta'}u(\hat{x})]\\
&-[f(\hat{y})+b(\hat{y})\cdot p+\mathcal{L}^{\delta'}(-\hat{\psi}(\hat{y}))+\mathcal{L}_{\delta'}v(\hat{y})]\\
&\leq\|\nabla f\|_\infty \cdot|d|+[b]_{1-\a}|d|^{{1-\a}}|p|+(\mathcal{L}^{\delta'}\hat{\phi}(\hat{x})-\mathcal{L}^{\delta'}(-\hat{\psi}(\hat{y})))+(\mathcal{L}_{\delta'}u(\hat{x})-\mathcal{L}_{\delta'}v(\hat{y}))\\
&= \|\nabla f\|_\infty \cdot|d|+[b]_{1-\a}|d|^{{1-\a}}|p|+r(\delta')+I_1+I_2, 
\end{split}
\end{equation}
here we denote 
$$
r(\delta'):= \mathcal{L}^{\delta'}\hat{\phi}(\hat{x})-\mathcal{L}^{\delta'}(-\hat{\psi}(\hat{y})), 
$$
$$I_1:=\int_{B^c_1}[(u(\hat{x}+z)-u(\hat{x}))-(v(\hat{y}+z)-v(\hat{y})]\mu(\d z), $$
and 
$$I_2:=\int_{B_1\cap B_{\delta'}^c} [(u(\hat{x}+z)-u(\hat{x}))-(v(\hat{y}+z)-
v(\hat{y}))]\mu(\d z). $$
By definition, we have
\begin{equation}\label{r}
\begin{split}
r(\delta')= & \mathcal{L}^{\delta'}\hat{\phi}(\hat{x})-\mathcal{L}^{\delta'}(-\hat{\psi}(\hat{y}))\\
=& CL\int_{|z|\leq \delta'}[|d+z|^\gamma+|d-z|^{\gamma}-2|d|^\gamma] \mu(\d z)\\
\leq& C(d)\int_{|z|\leq \delta'} |z|^2\mu(\d z)\rightarrow 0 \,\,(\delta' \to0). 
\end{split}
\end{equation}
Noticing $u(x)-v(y)-L|x-y|^\gamma$ reaches its maximal at point $(\hat{x}, \hat{y})$, we have 
$u(\hat{x})-v(\hat{y}) \geq u(\hat{x}+z)-v(\hat{y}+z)$. By this, we get 
\begin{equation}\label{I1}I_1\leq 0. \end{equation}

Now let $\delta:=|d|\eta\geq 2\delta'$ and  $S(\pmb{n}_d, \theta, \delta):=\{z: |z|\leq \delta, z\in S(\pmb{n}_d,\theta)\}$, denote 
$$I_2':=\int_{B_1\cap B_{\delta'}^c\cap S(\pmb{n}_d, \theta, \delta)} \left[(u(\hat{x}+z)-u(\hat{x})-p\cdot z)-(v(\hat{y}+z)-v(\hat{y})-p\cdot z)\right]\mu(\d z), $$
\begin{align*}I_2'':=\int_{B_1\cap B_{\delta'}^c\backslash S(\pmb{n}_d, \theta, \delta)} \left[(u(\hat{x}+z)-u(\hat{x}))-(v(\hat{y}+z)-v(\hat{y}))\right]\mu(\d z). \end{align*}
Recalling that $u(\hat{x})-v(\hat{y}) \geq u(\hat{x}+z)-v(\hat{y}+z)$, we get 
\begin{align}\label{I2''}
I_2''\leq 0. 
\end{align}
And since $u(\hat{x}+z)-v(\hat{y})-L|d+z|^\gamma\leq u(\hat{x})-v(\hat{y})-L|d|^\gamma$, we have
$$u(\hat{x}+z)-u(\hat{x})-p\cdot z\leq L|d+z|^\gamma-L|d|^\gamma-p\cdot z=L|d+z|^\gamma-L|d|^\gamma-\gamma L |d|^{\gamma-2}d\cdot z. $$
Similarly, we have 
$$-(v(\hat{y}+z)-v(\hat{y})-p\cdot z)\leq L|d-z|^\gamma-L|d|^\gamma+\gamma L |d|^{\gamma-2}d\cdot z. $$
For any $z\in S(\pmb{n}_d, \theta, \delta)$, by Taylor expansion,
\begin{align*}
&u(\hat{x}+z)-u(\hat{x})-p\cdot z\\
\leq&L|d+z|^\gamma-L|d|^\gamma-\gamma L |d|^{\gamma-2}d\cdot z\\
\leq &L\sup_{0<t<1}\left\{\gamma |d+tz|^{\gamma-2}|z|^2+\gamma(\gamma-2)|d+tz|^{\gamma-4}((d+tz)\cdot z)^2\right\}\\
=&L\sup_{0<t<1} \left\{\gamma|d+tz|^{\gamma-4}[|d+tz|^2|z|^2-(2-\gamma)((d+tz)\cdot z)^2]\right\}\\
\leq& -L\gamma|z|^2|d|^{\gamma-2}\left((2-\gamma)(1+\eta)^{\gamma-2}\cos^2\theta-(1-\eta)^{\gamma-2}\right). 
\end{align*}
Similarly, for any $z\in S(\pmb{n}_d, \theta, \delta)$, 
\begin{align*}
&-(v(\hat{y}+z)-v(\hat{y})-p\cdot z)\\
\leq&  L|d-z|^\gamma-L|d|^\gamma+\gamma L |d|^{\gamma-2}d\cdot z\\
\leq& -L\gamma|z|^2|d|^{\gamma-2}\left((2-\gamma)(1+\eta)^{\gamma-2}\cos^2\theta-(1-\eta)^{\gamma-2}\right). 
\end{align*}
Hence, by assumption {\bf (H$^\theta_\kappa$)}, 
\begin{equation}\label{I2'}
\begin{split}
I'_2&\leq -2L\gamma |d|^{\gamma-2}((1+\eta)^{\gamma-2}(2-\gamma)\cos^2\theta-(1-\eta)^{\gamma-2})\int_{B_{\delta'}^c\cap S(\pmb{n}_d, \theta, \delta)}|z|^2\mu(\d z)\\
&\leq -2L\gamma ((1+\eta)^{\gamma-2}(2-\gamma)\cos^2\theta-(1-\eta)^{\gamma-2})\cdot \kappa|d|^{\gamma-2}\delta^{2-\alpha}+R(\delta')\\
&= -\eps_0L|d|^{\gamma-\alpha}+R(\delta'),
\end{split}
\end{equation}
here $$\eps_0:=\eps_0(\gamma,\eta,\kappa,\alpha,\theta)=2 \kappa \eta^{2-\a}\gamma \{(1+\eta)^{\gamma-2}(2-\gamma)\cos^2\theta-(1-\eta)^{\gamma-2}\}>0,$$ 
and $R(\delta')\rightarrow 0 \,\,(\delta'\rightarrow 0)$. Combining \eqref{Sum}, \eqref{r}, \eqref{I1}, \eqref{I2''}, \eqref{I2'} and let $\delta'\rightarrow 0$, we obtain 
 
$$0<\lambda(M_L+M)\leq\|\nabla f\|_\infty \cdot|d|+L([b]_{1-\a}-\eps_0)|d|^{\gamma-\alpha}. $$
Recalling that $d=d(L)\rightarrow 0$ as $L\rightarrow \infty$. If $[b]_{1-\a}<\eps_0$ and $L$ sufficient large, we obtain 
$$\|\nabla f\|_\infty \cdot|d|>L(\eps_0-[b]_{1-\a})|d|^{\gamma-\alpha}, $$
which is a contradiction. So we complete the proof of our claim. 

Now assume $M>0$. For any $\eps>0$, let $$M'_\eps:=\sup_{x,y\in \R^d}\left\{u(x)-v(y)-\frac{|x-y|^2}{2\eps}\right\}>0,$$
by Lemma \ref{zero} and the argument above, we know that there exist $\bar{x}_\eps, \bar{y}_\eps\in\R^d$ such that $u(\bar{x}_\eps)-v(\bar{y}_\eps)-\frac{|\bar{x}_\eps-\bar{y}_\eps|^2}{2\eps}=M'_\eps>0$. Hence $$u(\bar{x}_\eps)-v(\bar{y}_\eps)-\frac{|\bar{x}_\eps-\bar{y}_\eps|^2}{2\eps}\geq \max_{x}\{u(x)-v(x)\}=M>0.$$
On the other hand, \eqref{Holder} implies
$$u(\bar{x}_\eps)-v(\bar{y}_\eps)-\frac{|\bar{x}_\eps-\bar{y}_\eps|^2}{2\eps}\leq M+L|\bar{x}_\eps-\bar{y}_\eps|^\gamma-\frac{|\bar{x}_\eps-\bar{y}_\eps|^2}{2\eps}$$
Therefor, 
\be\label{x-y}
|\bar{x}_\eps-\bar{y}_\eps|^{2-\gamma}\leq 2L\eps. 
\ee
Let $\bar{\phi}(x)=\frac{|x-\bar{y}_\eps|^2}{2\eps},\  \bar{\psi}(y)=\frac{|\bar{x}_\eps-y|^2}{2\eps}, \ \bar{p}:= \nabla \bar{\phi}(\bar{x}_\eps)=\nabla(-\bar{\psi}(\bar{y}_\eps))=\frac{\bar{x}_\eps-\bar{y}_\eps}{\eps}$. By the definition of viscosity sub(super)solution, 
\begin{equation*}
\begin{split}
0<\lambda M \leq& \lambda[u(\bar{x}_\eps)-v(\bar{y}_\eps)]\\
\leq & [f(\bar{x}_\eps)+b(\bar{x}_\eps)\cdot \bar{p}+\mathcal{L}^{\delta'}\bar{\phi}(\bar{x}_\eps)+\mathcal{L}_{\delta'}u(\bar{x}_\eps)]\\
&-[f(\bar{y}_\eps)+b(\bar{y}_\eps)\cdot \bar{p}+\mathcal{L}^{\delta'}(-\bar{\psi}(\bar{y}_\eps))+\mathcal{L}_{\delta'}v(\bar{y}_\eps)]\\
\leq& |f(\bar{x}_\eps)-f(\bar{y}_\eps)|+[b]_{1-\a}|\bar{x}_\eps-\bar{y}_\eps|^{{1-\a}}|\bar{p}|\\
&+(\mathcal{L}^{\delta'}\bar{\phi}(\bar{x}_\eps)-\mathcal{L}^{\delta'}(-\bar{\psi}(\bar{y}_\eps)))+(\mathcal{L}_{\delta'}u(\bar{x}_\eps)-\mathcal{L}_{\delta'}v(\bar{y}_\eps)). 
\end{split}
\end{equation*}
Just as the proof for \eqref{Holder} above, we can verify that the summation of the last two terms in the last inequality above is negative. So
$$0<\lambda M \leq |f(\bar{x}_\eps)-f(\bar{y}_\eps)|+ \frac{[b]_{1-\a} |\bar{x}_\eps-\bar{y}_\eps|^{2-\a}}{\eps}. $$
By \eqref{x-y} and noticing $2-\alpha>2-\gamma$, we get $|\bar{x}_\eps-\bar{y}_\eps|^{2-\a}/{\eps}\leq (2L\eps)^{(2-\a)/(2-\gg)}\rightarrow 0$, as $\eps\to 0$, which implies $M=0$. We complete the proof. 
\end{proof}
Now we give the proof of our main result.
\begin{proof}[Proof of Theorem \ref{Weak uniqueness}]
Suppose $\mu$ satisfies $({\bf{H}^\theta_\kappa})$ and $[b]_{1-\a}<\eps_0.$
Using lemma \ref{sub-super} and \ref{Comparison}, we get for any $\lambda>0$, $x\in\R^d$, $f\in C_c^1(\R^d)$ and $\P_x, \P'_x\in \Gamma_x$
$$\bE^{\P_x}\left[\int_0^\infty e^{-\lambda t}f(X_t)\d t\right]=\bE^{\P'_x}\left[\int_0^\infty e^{-\lambda t}f(X_t)\d t\right].$$
 By the uniqueness of Laplace transformation, we get for almost everywhere $t\geq 0$, $\P_x(X_t\in\cdot)=\P'_x(X_t\in\cdot)$. The right continuous of $X_t$ shows $\P_xX_t^{-1}=\P'_xX_t^{-1}$ for any $t\geq 0$. Using Theorem 4.4.2 of \cite{ethier2009markov}, weak uniqueness in this case follows. By standard localization technique(cf. \cite[Chapter 4]{ethier2009markov}), we get the uniqueness under assumption \eqref{b-condition}. 
\end{proof}

\section*{Acknowledgements}
The authors gratefully acknowledge Professor Zhen-Qing Chen for his great help and stimulating discussions. This work was supported by National Postdoctoral Program for Innovative Talents (201600182) of China.

\end{document}